\documentclass{article}

\usepackage{amsmath,amssymb,amsthm}
\usepackage{enumerate}
\usepackage{color}
\usepackage{authblk}

\title{Continued fraction algorithms and Lagrange's theorem in ${\mathbb Q}_p$}

\author[1,\thanks{Author to whom correspondence should be addressed. Electronic mail: saito@fun.ac.jp}]{Asaki Saito}
\author[2]{Jun-ichi Tamura}
\author[3]{Shin-ichi Yasutomi}
\affil[1]{Future University Hakodate, Hakodate, Hokkaido 041-8655, Japan}
\affil[2]{Tsuda College, Kodaira, Tokyo 187-8577, Japan}
\affil[3]{Toho University, Funabashi, Chiba 274-8510, Japan}

\newtheorem{thm}{Theorem}[section]
\newtheorem{cor}[thm]{Corollary}
\newtheorem{lem}[thm]{Lemma}

\theoremstyle{definition}

\theoremstyle{remark}
\newtheorem{rem}{Remark}[section]

\newcommand{\thmref}[1]{Theorem~\ref{#1}}

\newcommand{\lemref}[1]{Lemma~\ref{#1}}
\newcommand{\corref}[1]{Corollary~\ref{#1}}

\begin{document}
\maketitle

\begin{abstract}
We present several continued fraction algorithms, each of which gives an
eventually periodic expansion for every quadratic element of ${\mathbb
Q}_p$ over ${\mathbb Q}$ and gives a finite expansion for every
rational number.
We also give, for each of our algorithms, the complete
characterization of elements having purely periodic expansions.
\end{abstract}

\section{Introduction}

In this paper, we intend to add a classical flavor to the $p$-adic
world related to the well-known theorem of Lagrange (resp., Galois)
on the complete characterization of eventually (resp., purely)
periodic continued fractions (cf. \cite{Lagrange}, \cite{Galois}; see also \cite{Perron}).

In what follows, $p$ denotes a prime, ${\mathbb Q}_p$ the field of
$p$-adic numbers, and ${\mathbb Z}_p$ the ring of $p$-adic integers.
Schneider \cite{S} gave an algorithm that generates continued
fractions of the form
\begin{equation*}
  \cfrac{p^{k_{1}}}{d_{1}+
    \cfrac{p^{k_{2}}}{d_{2}+
      \cfrac{p^{k_{3}}}{
        \ddots}}} \quad \left(k_1 \in {\mathbb Z}_{\ge 0}, k_{n+1} \in {\mathbb Z}_{> 0},
  d_{n} \in \left\{1, \dots, p-1 \right\} ~(n \ge 1)\right)
\end{equation*}
and found periodic continued fractions for some quadratic
elements of ${\mathbb Z}_p$ over ${\mathbb Q}$ (see also
\cite{B}).
Ruban \cite{R} gave an algorithm that generates continued fractions of
the shape
\begin{equation*}
g_{0}+  \cfrac{1}{g_{1}+
    \cfrac{1}{g_{2}+
      \cfrac{1}{
        \ddots}}},
\end{equation*}
where
\begin{equation*}
g_n \in \left\{ \sum_{i=-m}^{0} e_{i} p^{i} ~\Bigg\vert~ m \in {\mathbb Z}_{\ge 0},
e_{i} \in \left\{0, 1, \dots, p-1 \right\}\right\}~(n \ge 0).
\end{equation*}
On the other hand, Weger \cite{W} has found a class of infinitely many
quadratic elements $\alpha \in {\mathbb Q}_p$ over ${\mathbb Q}$ such
that the continued fraction expansion of $\alpha$ obtained by
Schneider's algorithm is not periodic.
Ooto \cite{O} has found a similar result related to the algorithm given by
Ruban.
Weger \cite{Weger2} has mentioned in his paper, ``it seems that a
simple and satisfactory $p$-adic continued fraction algorithm does not
exist'', and given a periodicity result of lattices concerning
quadratic elements of ${\mathbb Q}_p$.
Browkin \cite{Browkin} has proposed some $p$-adic continued fraction algorithms;
nevertheless, the periodicity has not been proved for the
continued fractions obtained by applying his algorithms to quadratic
elements of ${\mathbb Q}_p$.
By disclosing a link between the hermitian canonical forms of certain
integral matrices and $p$-adic numbers, Tamura \cite{Tamura} has
shown that a multidimensional periodic continued fraction converges
to $\left(\alpha, \alpha^2,\dots,\alpha^{n-1} \right)$ in the $p$-adic
sense without considering algorithms of continued fraction expansion,
where $\alpha$ is the root of a polynomial in ${\mathbb Z}[X]$ of
degree $n$ stated in \lemref{HenselLemma}.
We can summarize the above situation as follows:
There have been proposed several $p$-adic continued fraction
algorithms.
However, it remains quite unclear whether or not there exists a simple
algorithm that generates periodic continued fractions for all the
algebraic elements of ${\mathbb Q}_p$ of fixed degree greater than
one.
Even for quadratic elements, $p$-adic versions of Lagrange's theorem
have not been found.

The main objectives of this paper are to define some
algorithms generating continued fractions of the form
\begin{equation}\label{eq:IntroInfiniteContinuedFraction}
d_{0}+ \cfrac{t_{1} p^{k_{1}}}{d_{1}+
    \cfrac{t_{2} p^{k_{2}}}{d_{2}+
      \cfrac{t_{3} p^{k_{3}}}{
     \ddots}}} \quad (k_{n} \in {\mathbb Z}_{> 0},
  ~ t_{n}, d_{n} \in {\mathbb Z} \setminus p{\mathbb Z} ~ (n \ge 1))
\end{equation}
with
$d_{0} \in {\mathbb Q}_p$ such that $d_{0}=\left[ d_{0} \right]$
(see \eqref{eq:DefinitionIntegralFractionalParts} for the definition
of the integral part $\left[ \alpha \right]$ of $\alpha \in {\mathbb
Q}_p$)
and to give
\begin{enumerate}[(i)]
\item $p$-adic versions of Lagrange's theorem for the three
  algorithms, i.e., the periodicity of the resulting continued
  fractions for quadratic elements of ${\mathbb Q}_p$ over ${\mathbb
  Q}$, and
\item $p$-adic versions of Galois' theorem concerning purely periodic
continued fractions.
\end{enumerate}
Moreover, we show that the continued fraction expansions of an
arbitrary rational number always terminate by our algorithms.

It is worth mentioning that our algorithms have a common background
with those proposed in \cite{TY1,TY2,TY3,TY4,FISTY} in the design of
continued fraction algorithms.

The rest of this paper is organized as follows.
In Section~\ref{sec:PadicContinuedFractionExpansion}, we consider
expanding $\alpha \in {\mathbb Q}_p$ into continued fractions whose
form is more general than the form \eqref{eq:IntroInfiniteContinuedFraction}.
In Section~\ref{sec:ConvergenceContinuedFractions}, we establish
convergence properties of the continued fractions introduced in
Section~\ref{sec:PadicContinuedFractionExpansion}.
In Section~\ref{sec:TwoBasicMaps}, we give two basic maps $T_{i}$
($i = 1, 2$) and present related lemmas.
We define three algorithms in terms of these basic
maps in Section~\ref{sec:ContinuedFractionAlgorithms}.
In Section~\ref{sec:ExpansionsQuadraticHenselRoots}, we show that each
of our algorithms gives an eventually periodic expansion for every
quadratic Hensel root, i.e., every quadratic element of ${\mathbb
Q}_p$ over ${\mathbb Q}$ whose existence is guaranteed by Hensel's
Lemma.
We do the same for every quadratic element in
Section~\ref{sec:ExpansionsQuadraticElementsQp}.
In Sections~\ref{sec:ExpansionsRationalHenselRoots} and
\ref{sec:ExpansionsRationalNumbers}, we show that the
continued fractions for every rational number obtained by our
algorithms always terminate.
We conclude with several remarks in
Section~\ref{sec:ConcludingRemarks}.

\section{$p$-adic continued fraction expansions}\label{sec:PadicContinuedFractionExpansion}

In what follows, $\alpha$ denotes an element of ${\mathbb Q}_p$ unless
otherwise mentioned.
We mean by the $p$-adic expansion of $\alpha$ the series
\begin{equation*}
\alpha=\sum_{i=-\infty}^{\infty} e_{i} p^{i} \quad \left(e_{i}=e_{i}(\alpha) \in \left\{0, 1, \dots, p-1 \right\}\right)
\end{equation*}
with $e_{i} \neq 0$ at most finitely many $i \le 0$.
We define the integral and fractional parts of $\alpha$, denoted by
$\left[ \alpha \right]$ and $\left\langle \alpha \right\rangle$
respectively, as
\begin{equation}\label{eq:DefinitionIntegralFractionalParts}
\left[ \alpha \right] := \sum_{i=-\infty}^{0} e_{i} p^{i}
\quad
\text{and} \quad \left\langle \alpha \right\rangle :=
\sum_{i=1}^{\infty} e_{i} p^{i}.
\end{equation}
In this section, we consider expanding $\alpha$ into a continued
fraction of the form
\begin{equation*}
d_{0}+ \cfrac{t_{1} p^{k_{1}}}{d_{1}+
    \cfrac{t_{2} p^{k_{2}}}{d_{2}+
      \cfrac{t_{3} p^{k_{3}}}{
     \ddots}}} \quad (k_{i} \in {\mathbb Z}_{> 0},
  ~ t_{i}, d_{i} \in {\mathbb Z}_p \setminus p{\mathbb Z}_p ~ (i \ge 1))
\end{equation*}
with
$d_{0} \in {\mathbb Q}_p$ such that $d_{0}=\left[ d_{0} \right]$.
Note that this class of continued fractions contains ones of the form
\eqref{eq:IntroInfiniteContinuedFraction}.

Let $t$ be a map from $p{\mathbb Z}_p \setminus \left\{ 0 \right\}$ to
${\mathbb Z}_p \setminus p{\mathbb Z}_p$.
Then, $\frac{t(x) p^{v_p(x)}}{x} \in {\mathbb Z}_p \setminus p{\mathbb
Z}_p$ for all $x \in p{\mathbb Z}_p \setminus \left\{ 0 \right\}$,
where $v_p(\alpha)$ denotes the
$p$-adic additive valuation of $\alpha$.
We consider a family of the maps of the form 
\begin{equation}\label{eq:T}
\begin{split}
&T: p{\mathbb Z}_p \setminus \left\{ 0 \right\} \rightarrow
p{\mathbb Z}_p,\\
&T(x):=\frac{t(x) p^{v_p(x)}}{x} - d(x),
\end{split}
\end{equation}
where $d$ is a map from $p{\mathbb Z}_p \setminus \left\{ 0 \right\}$
to ${\mathbb Z}_p \setminus p{\mathbb Z}_p$.
Since $d(x)=\frac{t(x) p^{v_p(x)}}{x} - T(x)$ and $T(x) \in p{\mathbb
Z}_p$, we have $\left[ d(x) \right] = \left[ \frac{t(x)
p^{v_p(x)}}{x} \right] \in \left\{1, \dots, p-1 \right\}$ for
all $x \in p{\mathbb Z}_p \setminus \left\{ 0 \right\}$.
Hence, $d$ is uniquely determined if the image of $d$, denoted by
$\text{Im}(d)$, satisfies
$\text{Im}(d) \subset \left\{1, \dots, p-1 \right\}$.

Since 
\begin{equation*}
x=\frac{t(x) p^{v_p(x)}}{d(x) + T(x)},
\end{equation*}
we have
\begin{equation*}
T^{n-1}(\left\langle \alpha \right\rangle)=\frac{t(T^{n-1}(\left\langle \alpha \right\rangle)) p^{v_p(T^{n-1}(\left\langle \alpha \right\rangle))}}{d(T^{n-1}(\left\langle \alpha \right\rangle)) + T^{n}(\left\langle \alpha \right\rangle)},
\end{equation*}
provided that $T^{n-1}(\left\langle \alpha \right\rangle) \neq 0$ ($n
\in {\mathbb Z}_{> 0}$).
Setting
\begin{align*}
t_{i} &=t\left( T^{i-1}(\left\langle \alpha \right\rangle) \right),\\
k_{i} &=v_p\left( T^{i-1}(\left\langle \alpha \right\rangle) \right),\\
d_{i} &=d\left( T^{i-1}(\left\langle \alpha \right\rangle) \right),
\end{align*}
for $i \in \left\{1, \dots, n \right\}$, we have
\begin{equation*}
\alpha=\left[ \alpha \right]+
  \cfrac{t_{1} p^{k_{1}}}{d_{1}+
    \cfrac{t_{2} p^{k_{2}}}{d_{2}+
      \cfrac{t_{3} p^{k_{3}}}{\quad\ddots
        \rule{0pt}{6ex} +
       \cfrac{t_{n-1} p^{k_{n-1}}}{d_{n-1}+
        \cfrac{t_{n} p^{k_{n}}}{d_{n}+ T^{n}(\left\langle \alpha \right\rangle)}
  }}}}.
\end{equation*}

Related to the continued fraction expansion of $\alpha$, there occur
three cases:
\begin{enumerate}[(i)]
\item $\left\langle \alpha \right\rangle = 0$.

We do not expand $\left\langle \alpha \right\rangle = 0$, and we have
$\alpha=\left[ \alpha \right]$.  

\item There exists $N \in {\mathbb Z}_{> 0}$ such that
$T^{N}(\left\langle \alpha \right\rangle) = 0$
and
$T^{n}(\left\langle \alpha \right\rangle) \neq 0$ for all $0 \le n
<N$.

We can expand $\alpha$ into the finite continued fraction
\begin{equation}\label{eq:FiniteContinuedFraction}
\alpha=\left[ \alpha \right]+
  \cfrac{t_{1} p^{k_{1}}}{d_{1}+
    \cfrac{t_{2} p^{k_{2}}}{d_{2}+
      \cfrac{t_{3} p^{k_{3}}}{\quad\ddots
       \rule{0pt}{6ex} +\cfrac{t_{N} p^{k_{N}}}{d_{N}}
  }}}.
\end{equation}

\item $T^{n}(\left\langle \alpha \right\rangle) \neq 0$ for all $n \in
  {\mathbb Z}_{\ge 0}$.

We can expand $\alpha$ into the infinite continued fraction
\begin{equation}\label{eq:InfiniteContinuedFraction}
\left[ \alpha \right]+
  \cfrac{t_{1} p^{k_{1}}}{d_{1}+
    \cfrac{t_{2} p^{k_{2}}}{d_{2}+
      \cfrac{t_{3} p^{k_{3}}}{
     \ddots}}}.
\end{equation}
We will show that the continued fraction
\eqref{eq:InfiniteContinuedFraction} converges to $\alpha$ in the succeeding
section.
\end{enumerate}

\begin{rem}
~
\begin{enumerate}[(i)]
\item We can consider a variety of maps $T$.
In fact, we will give three algorithms of continued fraction expansion
in Section~\ref{sec:ContinuedFractionAlgorithms};
consequently, the expression in continued fractions
\eqref{eq:FiniteContinuedFraction} and
\eqref{eq:InfiniteContinuedFraction} are not uniquely determined for a
given $\alpha \in {\mathbb Q}_p$.

\item The continued fractions considered by Schneider \cite{S} are
generated by setting $t(x) \equiv 1$ and
$\text{Im}(d) \subset \left\{1, \dots, p-1
\right\}$.
\end{enumerate}
\end{rem}

\section{Convergence of continued fractions}\label{sec:ConvergenceContinuedFractions}

In this section, we show that the continued fraction described in
Section~\ref{sec:PadicContinuedFractionExpansion} always converges to
$\alpha$ for $\alpha \in {\mathbb Q}_p$.
Without loss of generality, we may assume that $\alpha \in p{\mathbb
Z}_p \setminus \left\{ 0 \right\}$ in this section.

We define two sequences $\left\{ p_{n} \right\}_{n \ge -1}$ and
$\left\{ q_{n} \right\}_{n \ge -1}$ in terms of $t_{i}$, $k_{i}$,
$d_{i}$ in \eqref{eq:InfiniteContinuedFraction}
by the following recursion
formulas:
\begin{equation*}
\left\{
\begin{aligned}
&p_{-1}=1,& &p_{0}=0,& &p_{n}=d_{n}p_{n-1}+t_{n} p^{k_{n}} p_{n-2}& &(n \ge 1),\\
&q_{-1}=0,& &q_{0}=1,& &q_{n}=d_{n}q_{n-1}+t_{n} p^{k_{n}} q_{n-2}& &(n \ge 1).
\end{aligned}
\right.
\end{equation*}
In the case of the finite expansion
\eqref{eq:FiniteContinuedFraction}, we define $p_{n}$ and $q_{n}$ for
$n$ with $-1 \le n \le N$.

Lemmas~\ref{RationalApproximation}--\ref{PQ_MINUS_PQ} given below are
easily seen (cf. \cite{Perron}).

\begin{lem}\label{RationalApproximation}
\begin{equation*}
\cfrac{t_{1} p^{k_{1}}}{d_{1}+
    \cfrac{t_{2} p^{k_{2}}}{d_{2}+
      \cfrac{t_{3} p^{k_{3}}}{\quad\ddots
       \rule{0pt}{6ex} +\cfrac{t_{n} p^{k_{n}}}{d_{n}}
}}}
=
\frac{p_{n}}{q_{n}} \quad \left(n \ge 1\right).
\end{equation*}
\end{lem}


\begin{lem}\label{AlphaAndTn}
\begin{equation*}
\alpha = \frac{p_{n} + T^{n}(\alpha) p_{n-1}}{q_{n} + T^{n}(\alpha) q_{n-1}} \quad \left(n \ge 1\right).
\end{equation*}
\end{lem}


\begin{lem}\label{PQ_MINUS_PQ}
\begin{equation*}
p_{n-1} q_{n} - p_{n} q_{n-1} = \prod_{i=1}^{n} \left(- t_{i} p^{k_{i}} \right) \quad \left(n \ge 1\right).
\end{equation*}
\end{lem}


We denote by $\lvert \alpha \rvert_p$ the $p$-adic absolute value of
$\alpha \in {\mathbb Q}_p$, i.e., $\lvert \alpha \rvert_p := 1/p^{v_p(\alpha)}$.

\begin{lem}\label{AbsoluteValueQ}
\begin{equation*}
\lvert q_{n} \rvert_p = 1 \quad \left(n \ge 0\right).
\end{equation*}
\end{lem}

\begin{proof}
The claim is true for $n = 0$ and $n = 1$.
Assuming that $\lvert q_{i} \rvert_p = 1$ holds for $0 \le i \le n$
with $n \ge 1$,
we have $\lvert q_{n+1} \rvert_p = \lvert d_{n+1}q_{n}+t_{n+1}
p^{k_{n+1}} q_{n-1} \rvert_p =1$ since $\lvert d_{n+1}q_{n} \rvert_p =
1$ and $\lvert t_{n+1} p^{k_{n+1}} q_{n-1} \rvert_p \le 1/p$.
\end{proof}

\begin{thm}\label{ConvergenceTheorem}
~
\begin{enumerate}[(i)]
\item Let $n$ be an integer with $n \ge 1$ or an integer with $1 \le n
\le N$ if there exists an integer $N \ge 1$ such that $T^{N}( \alpha
) = 0$.
Then,
\begin{equation*}
\left\lvert \alpha - \frac{p_{n}}{q_{n}} \right\rvert_p =
\frac{\left\lvert T^{n}(\alpha) \right\rvert_p }{p^{\sum_{i=1}^{n} k_{i}}}
\end{equation*}
holds.
In particular,
\begin{equation*}
\left\lvert \alpha - \frac{p_{n}}{q_{n}} \right\rvert_p =
\frac{ 1 }{p^{\sum_{i=1}^{n+1} k_{i}}}
\end{equation*}
holds if $T^{n}(\alpha) \neq 0$.

\item Let $T^{n}( \alpha ) \neq 0$ for all $n \ge 1$.
Then,
\begin{equation*}
\lim_{n \to \infty} \frac{p_{n}}{q_{n}}  = \alpha
\end{equation*}
holds.
\end{enumerate}
\end{thm}

\begin{proof}
(i) By \lemref{AlphaAndTn}, we have
\begin{equation*}
\alpha - \frac{p_{n}}{q_{n}} =
\frac{p_{n} + T^{n}(\alpha) p_{n-1}}{q_{n} + T^{n}(\alpha) q_{n-1}} - \frac{p_{n}}{q_{n}} =
\frac{T^{n}(\alpha) \left( p_{n-1} q_{n} - p_{n} q_{n-1} \right)}{\left( q_{n} + T^{n}(\alpha) q_{n-1} \right) q_{n}}.
\end{equation*}
By \lemref{PQ_MINUS_PQ}, we have
\begin{equation*}
\left\lvert T^{n}(\alpha) \left( p_{n-1} q_{n} - p_{n} q_{n-1} \right) \right\rvert_p
= \left\lvert T^{n}(\alpha) \prod_{i=1}^{n} \left(- t_{i} p^{k_{i}} \right) \right\rvert_p
= \frac{\left\lvert T^{n}(\alpha) \right\rvert_p }{p^{\sum_{i=1}^{n} k_{i}}}.
\end{equation*}
By \lemref{AbsoluteValueQ}, we have
\begin{equation*}
\left\lvert \left( q_{n} + T^{n}(\alpha) q_{n-1} \right) q_{n} \right\rvert_p = 1.
\end{equation*}
Hence, we get
\begin{equation*}
\left\lvert \alpha - \frac{p_{n}}{q_{n}} \right\rvert_p =
\frac{\left\lvert T^{n}(\alpha) \right\rvert_p }{p^{\sum_{i=1}^{n} k_{i}}}.
\end{equation*}
If $T^{n}(\alpha) \neq 0$, then $\left\lvert T^{n}(\alpha)
\right\rvert_p = 1 / p^{k_{n+1}}$, which implies
\begin{equation*}
\left\lvert \alpha - \frac{p_{n}}{q_{n}} \right\rvert_p =
\frac{ 1 }{p^{\sum_{i=1}^{n+1} k_{i}}}.
\end{equation*}
The assertion (ii) immediately follows from (i).
\end{proof}

\section{Two basic maps: $T_{1}$ and $T_{2}$}\label{sec:TwoBasicMaps}

We later propose three continued fraction algorithms, each of which
gives an eventually periodic expansion for every quadratic element of
${\mathbb Q}_p$ over ${\mathbb Q}$ and gives a finite expansion for
every rational number.
In this section, we introduce maps $T_{1}$ and $T_{2}$ on the basis of
which we construct the algorithms.

We denote by $A_p$ the set of all the elements of $p{\mathbb Z}_p$
which are algebraic over ${\mathbb Q}$ of degree at most two.
For simplicity, we will abbreviate ``algebraic over ${\mathbb Q}$'' to
``algebraic'', and ``quadratic over ${\mathbb Q}$'' to ``quadratic''.
We mean, by the minimal polynomial of an algebraic element
$\alpha$, the integral polynomial of the lowest degree which
has $\alpha$ as a root, whose leading coefficient is positive, and
whose coefficients are coprime.
We denote the minimal polynomial of $x \in A_p$ by $aX^2 + bX + c$ if $x$ is quadratic.
We denote it by $bX + c$ if $x$ is rational.
Note that $c \neq 0$ if and only if $x \neq 0$.
Let us define a map $u: A_p \setminus \left\{ 0 \right\} \rightarrow
{\mathbb Z} \setminus p{\mathbb Z}$ by assigning
\begin{equation*}
u(x):= c \lvert c
\rvert_p \in {\mathbb Z} \setminus p{\mathbb Z}
\end{equation*}
to each $x \in A_p
\setminus \left\{ 0 \right\}$.
We define two maps $T_{1}$ and $T_{2}$ from $A_p \setminus \left\{ 0
\right\}$ to $A_p$ by
\begin{align*}
T_{1}(x) &:=\frac{u(x) p^{v_p(x)}}{x} - d_{1}(x),\\
\intertext{and}
T_{2}(x) &:=\frac{-u(x) p^{v_p(x)}}{x} - d_{2}(x),
\end{align*}
where $d_{1}$ and $d_{2}$ are maps from $A_p \setminus \left\{ 0
\right\}$ to $\left\{1, \dots, p-1 \right\}$ which are
uniquely defined so as to let $T_{1}(x)$ and $T_{2}(x)$ belong to $p{\mathbb
Z}_p$, and thus to $A_p$, for every $x \in A_p \setminus \left\{ 0 \right\}$.
It is clear that $T_{1}$ and $T_{2}$ map any quadratic element of $A_p$
to a quadratic one.
We remark that $T_{1}$ and $T_{2}$ belong to the family of the maps
\eqref{eq:T} if we ignore their domains.

Our algorithms introduced in the next section reduce expansions of
algebraic elements of ${\mathbb Q}_p$ of degree at most two
to expansions of those of $p{\mathbb Z}_p$ whose
existence is guaranteed by the following well-known lemma.

\begin{lem}[Hensel's Lemma]\label{HenselLemma}
Let $f(X):=X^{n} + a_{n-1} X^{n-1} + \dots + a_{1} X + a_{0} \in {\mathbb Z}[X]$,
where $n \in {\mathbb Z}_{> 0}$, $a_{1} \in {\mathbb Z} \setminus p{\mathbb Z}$,
and $a_{0} \in p{\mathbb Z}$.
Then, there exists unique $\alpha \in p{\mathbb Z}_p$ such that
$f(\alpha)=0$.
\end{lem}

In what follows, we call an element of $p{\mathbb Z}_p$ a {\it
quadratic Hensel root} if it is a root of $X^{2} + b X + c \in
{\mathbb Z}[X]$ where $b \in {\mathbb Z} \setminus p{\mathbb Z}$, $c
\in p{\mathbb Z}$, and $X^{2} + b X + c$ is irreducible.
Likewise, we call an element of $p{\mathbb Z}_p$ a {\it rational
Hensel root} if it is a root of $X + c \in {\mathbb Z}[X]$ with $c
\in p{\mathbb Z}$ (obviously, the root is $-c$).

\begin{lem}\label{T1AndT2}
$T_{1}$ and $T_{2}$ map every quadratic Hensel root to a quadratic
Hensel root.
\end{lem}

\begin{proof}
Let $\alpha \in p{\mathbb Z}_p$ be an arbitrary quadratic Hensel root.
By definition, $\alpha$ has a minimal polynomial of the form $X^{2} +
b X + c \in {\mathbb Z}[X]$ with $b \in {\mathbb Z} \setminus
p{\mathbb Z}$ and $c \in p{\mathbb Z}$.
We see that the conjugate $\alpha^{\sigma} \neq \alpha$ of $\alpha$ satisfies
$\alpha^{\sigma}\in {\mathbb Z}_p \setminus p{\mathbb Z}_p$ since
$\alpha^{\sigma}= -b -\alpha$.
Since $\alpha=c/\alpha^{\sigma}$, we have $\lvert \alpha \rvert_p=
\lvert c \rvert_p$.
Recalling the definition of $u$, we see that
\begin{equation*}
\frac{u(\alpha) p^{v_p(\alpha)}}{\alpha}=\frac{c}{\alpha}=\alpha^{\sigma}.
\end{equation*}
Let
\begin{equation*}
  \alpha^{\sigma}=\sum_{i=0}^{\infty} e_{i} p^{i}
  \quad \left( e_{i} \in \left\{0, 1, \dots, p-1 \right\}, \quad e_{0} \neq 0 \right).
\end{equation*}
Since $\alpha^{\sigma}$ is a root of $X^{2} + b X + c$, we have $e_{0}
\left( e_{0} + b \right) \equiv 0 \pmod{p}$.
Let $r$ be an element of $\left\{1, \dots, p-1 \right\}$ satisfying $r
\equiv b \pmod{p}$.
Since $e_{0} \neq 0$, we have $e_{0} + b \equiv 0 \pmod{p}$, which
implies $e_{0} = p-r$.
Thus, $d_{1}(\alpha)=p-r$, and we have $T_{1}(\alpha)=\alpha^{\sigma}
- (p-r) \in p{\mathbb Z}_p$.
By substituting $X + (p-r)$ for $X$ in $X^{2} + b X + c$, we have the
minimal polynomial of $T_{1}(\alpha)$ given by
\begin{equation}\label{eq:MinimalPolynomialT1}
X^{2} + \left\{ b + 2 (p-r) \right\} X + (p-r) b + c + (p-r)^2 \in {\mathbb Z}[X]. 
\end{equation}
Since ${\mathbb Z} \setminus p{\mathbb Z} \ni b + 2 (p-r) \equiv -r
\pmod{p}$ and $(p-r) b + c + (p-r)^2 \in p{\mathbb Z}$, we see that
$T_{1}(\alpha)$ is a quadratic Hensel root.

By a similar argument, we can show that
$T_{2}(\alpha)=-\alpha^{\sigma} - r \in p{\mathbb Z}_p$ and its
minimal polynomial is given by
\begin{equation}\label{eq:MinimalPolynomialT2}
X^{2} + \left( -b + 2r \right) X  -rb + c + r^2 \in {\mathbb Z}[X]. 
\end{equation}  
Since $-b + 2r \in {\mathbb Z} \setminus p{\mathbb Z}$ and $-rb + c +
r^2 \in p{\mathbb Z}$, we see that $T_{2}(\alpha)$ is also a quadratic
Hensel root.
\end{proof}

\begin{rem}
Both maps $T_{1}$ and $T_{2}$ preserve discriminants of
the minimal polynomials of quadratic Hensel roots,
i.e., the discriminants of \eqref{eq:MinimalPolynomialT1} and
\eqref{eq:MinimalPolynomialT2}
are equal to $b^2 -4c$.
\end{rem}

\section{Continued fraction algorithms}\label{sec:ContinuedFractionAlgorithms}

On the basis of $T_{1}$ and $T_{2}$ introduced in the previous
section, we can consider a variety of continued fraction algorithms
which yield an eventually periodic expansion for every quadratic
element of ${\mathbb Q}_p$ and yield a finite expansion for every
rational number.
In the present paper, we deal with three particular algorithms.
As in Section~\ref{sec:TwoBasicMaps}, the minimal polynomial of $x \in
A_p$ is denoted by $aX^2 + bX + c \in
{\mathbb Z}[X]$ for quadratic $x$, and by $bX + c \in {\mathbb Z}[X]$
for rational $x$.
Our algorithms decide which map, $T_{1}$ or $T_{2}$, is applied to a
given $x \in A_p \setminus \left\{ 0 \right\}$ on the basis of two
coefficients of its minimal polynomial, namely the coefficient $b$ of
$X$ and the constant term $c$, regardless of the degree of $x$.
In the following, we specify our algorithms by specifying the map $T:
A_p \setminus \left\{ 0 \right\} \rightarrow A_p$ used by each algorithm:

\begin{align*}
\intertext{\bf Algorithm A:}
T(x) &:=T_{2}(x).\\
\intertext{\bf Algorithm B:}
T(x) &:=
\begin{cases}
T_{2}(x)& \text{if $b \ge 0$},\\
T_{1}(x)& \text{if $b < 0$}.
\end{cases}\\
\intertext{\bf Algorithm C:}
T(x) &:=
\begin{cases}
T_{2}(x)& \text{if $b \ge 0$ and $c > 0$},\\
T_{1}(x)& \text{otherwise}.
\end{cases}
\end{align*}

\section{Expansions of quadratic Hensel roots}\label{sec:ExpansionsQuadraticHenselRoots}

In this section, we deal with the expansions of quadratic Hensel
roots, on the basis of which we expand general quadratic elements of
${\mathbb Q}_p$.
We show that each of our algorithms gives an eventually
periodic expansion for any quadratic Hensel root.
We will deal with the expansions of general quadratic elements of ${\mathbb
Q}_p$ and those of rational
numbers in the subsequent sections.

When considering an expansion of a quadratic Hensel root $\alpha$, it
is convenient to identify $\alpha$ with the pair $(b, c)$ of
coefficients of its minimal polynomial $X^{2} + b X + c$.
In the following, we will do so and allow writing $\alpha=(b,
c)$.
Similarly, we write $T(\alpha)$ also as $T(b, c)$.
We note that in view of \eqref{eq:MinimalPolynomialT1} and
\eqref{eq:MinimalPolynomialT2}, we have
\begin{align}
T_{1}(b, c) &=\left(b + 2 (p-r), (p-r) b + c + (p-r)^2\right),\label{eq:QuadraticHenselT1}\\
\intertext{and}
T_{2}(b, c) &=\left(-b + 2r, -rb + c + r^2\right),\label{eq:QuadraticHenselT2}
\end{align}
where $r$ is an element of $\left\{1, \dots, p-1 \right\}$ satisfying
$r \equiv b \pmod{p}$.

Let $S$ be the set of all quadratic Hensel roots, i.e.,
\begin{equation*}
S=\left\{ (b, c) \in {\mathbb Z}^2 ~\vert~ b \in {\mathbb Z} \setminus p{\mathbb Z}, ~c \in p{\mathbb Z}, \text{ and } X^{2} + b X + c \text{ is irreducible}\right\}.
\end{equation*}  
We put
\begin{align*}
  S_1 &:= \left\{ (b, c) \in S ~\vert~ b>0, ~c>0 \right\},\\
  S_2 &:= \left\{ (b, c) \in S ~\vert~ b<0, ~c>0 \right\},\\
  S_3 &:= \left\{ (b, c) \in S ~\vert~ b<0, ~c<0 \right\},\\
  S_4 &:= \left\{ (b, c) \in S ~\vert~ b>0, ~c<0 \right\}.
\end{align*}  
We further put
\begin{align*}
  R   &:= \left\{ (b, c) \in S  ~\vert~ 1 \le b \le p-1 \right\},\\
  R_1 &:= \left\{ (b, c) \in S_1 ~\vert~ 1 \le b \le p-1 \right\},\\
  R_4 &:= \left\{ (b, c) \in S_4 ~\vert~ 1 \le b \le p-1 \right\}.
\end{align*}

In the following subsections, we give
Theorems~\ref{QuadraticHenselAlgorithmA},
\ref{QuadraticHenselAlgorithmB}, and \ref{QuadraticHenselAlgorithmC}
which state the periodicity of the continued fraction expansion
obtained by Algorithms A, B, and C, for any given quadratic Hensel
root.

\subsection{Expansions of quadratic Hensel roots by Algorithm~A}
\label{sec:QuadraticHenselRootsAlgorithmA}

\begin{thm}\label{QuadraticHenselAlgorithmA}
The expansion of every quadratic Hensel root obtained by Algorithm A
(i.e., $T_{2}$) is purely periodic with period one or two.
\end{thm}

\begin{proof}
Let $(b, c) \in S$.
Let $r$ be an element of $\left\{1, \dots, p-1 \right\}$ satisfying $r
\equiv b \pmod{p}$.
Then,
$T_{2}(b, c) =\left(-b + 2r, -rb + c + r^2\right)$.
Using $-b + 2r \equiv r \pmod{p}$, we easily see that
$T_{2}^{2}(b, c) =\left(b, c\right)$.
Thus, $(b, c)$ is a purely periodic point with period two or one.
\end{proof}

\begin{rem}\label{FixedPointT2}
It is easy to see that $(b, c) \in S$ is a fixed
point of $T_{2}$ if and only if $(b, c) \in R$.
\end{rem}

\subsection{Expansions of quadratic Hensel roots by Algorithm~B}

\begin{thm}\label{QuadraticHenselAlgorithmB}
The expansion of every quadratic Hensel root obtained by Algorithm B
is eventually periodic with period one.
\end{thm}

\begin{proof}
Let $(b, c) \in S$.
We will show that $(b, c)$ is an eventually fixed point of the map $T$
associated with Algorithm B by considering the following three cases:
\begin{enumerate}[(i)]
\item $b \in \left\{1, \dots, p-1 \right\}$.
According to the definition of Algorithm B,
we apply $T_{2}$ to $(b, c) \in R$.
As described in Remark~\ref{FixedPointT2}, such $(b, c)$ is a
fixed point.

\item $b<0$.
We apply $T_{1}$ to $(b, c)$ with $b<0$.
We can write $b=-np+r$, where $n \in {\mathbb Z}_{> 0}$ and $r \in
\left\{1, \dots, p-1 \right\}$.
Let $(b', c')=T_{1}(b, c)$.
Then, $b'=-(n-1)p+p-r$.
Thus, the $n$-fold iteration of $T_{1}$ maps $(b, c)$ to a
fixed point given in (i).

\item $b>p$.
We apply $T_{2}$ to $(b, c)$ with $b>p$.
We can write $b=np+r$, where $n \in {\mathbb Z}_{> 0}$ and $r \in
\left\{1, \dots, p-1 \right\}$.
Let $(b', c')=T_{2}(b, c)$.
Since $b'=-np+r<0$, this case reduces to (ii).
\end{enumerate}
\end{proof}

By the proof of \thmref{QuadraticHenselAlgorithmB}, we get the
following corollary.

\begin{cor}\label{PurelyPeriodicPointsAlgorithmB}
The set of purely periodic points of $T$ associated with Algorithm B
in $S$ is $R$.
\end{cor}

\subsection{Expansions of quadratic Hensel roots by Algorithm~C}

\begin{thm}\label{QuadraticHenselAlgorithmC}
The expansion of every quadratic Hensel root obtained by Algorithm C
is eventually periodic.
\end{thm}

\begin{proof}
We will show that every orbit of $T$ associated with Algorithm C
starting from a quadratic Hensel root is eventually periodic.

First, we need to discuss the dynamics of $T$ on $S$.
We apply $T_{2}$ to $(b, c) \in S_1$ and $T_{1}$ to $(b, c) \in
\bigcup_{i=2}^4 S_i$.
We see by \eqref{eq:QuadraticHenselT1} that there exists no fixed
point of $T_{1}$ in $\bigcup_{i=2}^4 S_i$ since $b \neq b + 2 (p-r)$.
Thus, the fixed points of $T$ on $S$ are those of $T_{2}$ in $S_1$,
i.e., the points $(b, c) \in R_1$ (cf. Remark~\ref{FixedPointT2}).
We see by \eqref{eq:QuadraticHenselT1} that in $S_4$, the values of
$b$ and $c$ strictly increase with each iteration of $T_{1}$, and
thus every $(b, c) \in S_4$ is eventually mapped into $S_1$.
Every $(b, c) \in S_1$ other than the fixed points is mapped into
either $S_2$ or $S_3$ under $T_{2}$ (cf. Proof (iii) of
\thmref{QuadraticHenselAlgorithmB}).
In $S_2$ and $S_3$, the value of $b$ strictly increases with each
iteration of $T_{1}$, and every $(b, c) \in S_2 \cup S_3$ is
eventually mapped into $R = R_1 \cup R_4$ (cf. Proof (ii) of
\thmref{QuadraticHenselAlgorithmB}).

Second, we should note that $T_{1}$ on $S$ is bijective.
The inverse map $T_{1}^{-1}: S \rightarrow S$ is given by
\begin{equation}\label{eq:QuadraticHenselT1Inverse}
T_{1}^{-1}(b, c) =\left(b - 2 r, -r b + c + r^2\right),
\end{equation}
where $r$ is an element of $\left\{1, \dots, p-1 \right\}$ satisfying
$r \equiv b \pmod{p}$.

Due to the dynamics of $T$ on $S$, any orbit starting from a quadratic
Hensel root eventually enters either $R_1$ or $R_4$.
If the orbit enters $R_1$, then the orbit is eventually periodic with
period one since every element of $R_1$ is a fixed point.

In what follows, we will show that the orbit entering $R_4$ is also
eventually periodic by showing that every element of $R_4$ is a
purely periodic point.
Let $(b_0, c_0) \in R_4$.
Repeated iteration of $T$ (i.e., $T_{1}$) eventually maps $(b_0, c_0)$
into $S_1$.
Two cases occur:
\begin{enumerate}[(i)]
\item $(b_0, c_0)$ is mapped into $S_1$ by iterating $T$ even times,
i.e., there exists $n \in {\mathbb Z}_{> 0}$ such that $T^{2n}(b_0,
c_0) \in S_1$ and $T^{i}(b_0, c_0) \in S_4$ for $0 \le i \le 2n-1$.

\item $(b_0, c_0)$ is mapped into $S_1$ by iterating $T$ odd times,
i.e., there exists $n \in {\mathbb Z}_{> 0}$ such that $T^{2n-1}(b_0,
c_0) \in S_1$ and $T^{i}(b_0, c_0) \in S_4$ for $0 \le i \le 2n-2$.
\end{enumerate}

We note the following fact: Let
$r:=b_0 \in \left\{1, \dots, p-1 \right\}$.
By induction, we can show
\begin{align*}
T_{1}^{2k}(b_0, c_0) &=(2 p k + r,~ p^2 k^2 + r p k + c_0) &(k \in {\mathbb Z}),\\
T_{1}^{2k-1}(b_0, c_0) &=(2 p k - r,~ p^2 k^2 - r p k + c_0) &(k \in {\mathbb Z}).
\end{align*}

Let us consider Case (i). We see that
\begin{align*}
  T^{2n}(b_0, c_0) &= T_{1}^{2n}(b_0, c_0)\\
  &=(2 p n + r,~ p^2 n^2 + r p n + c_0) \in S_1.
\end{align*}
Hence, we have
\begin{align*}
  T^{2n+1}(b_0, c_0) &= T_{2}(2 p n + r,~ p^2 n^2 + r p n + c_0)\\
  &=(-2 p n + r,~ p^2 n^2 - r p n + c_0).
\end{align*}
On the other hand, we see that
\begin{align*}
  T_{1}^{-2n}(b_0, c_0)  &=(-2 p n + r,~ p^2 n^2 - r p n + c_0)\\
  &= T^{2n+1}(b_0, c_0).
\end{align*}
Since $-2 p n + r < 0$,
we have
\begin{align*}
  (b_0, c_0)  &=T_{1}^{2n} \circ T^{2n+1}(b_0, c_0)\\
  &= T^{4n+1}(b_0, c_0).
\end{align*}
Therefore, in Case (i), $(b_0, c_0)$ is a purely periodic point.

In a similar manner, we can prove that also in Case (ii), $(b_0, c_0)$
is a purely periodic point.
\end{proof}


The following lemma characterizes the set of purely periodic points
of $T$ associated with Algorithm C in $S$.

\begin{lem}\label{PurelyPeriodicPointsAlgorithmC}
The set of purely periodic points of $T$ associated with Algorithm C
in $S$ is $P_1 \cup R_1 \cup S_3 \cup S_4$, where $P_1$ is defined by
\begin{equation*}
P_1 := \left\{ (b, c) \in S_1 \setminus R_1 ~\vert~ c < \left[ b
\right] \left\langle b \right\rangle \right\}.
\end{equation*}  
\end{lem}

\begin{proof}
Every element of $R_1$ and $R_4$ is a purely periodic point
(cf. the proof of \thmref{QuadraticHenselAlgorithmC}).

Every $(b, c) \in S_4 \setminus R_4$ is a purely periodic point since
$(b, c)$ is mapped into $R_4$ by iterating $T_{1}^{-1}$.
Thus, every element of $S_4$ is a purely periodic point.

It is not difficult to see that $(b, c) \in S_1 \setminus R_1$ is a
purely periodic point if and only if $T_{1}^{-1}(b, c) \in S_4$ which,
by \eqref{eq:QuadraticHenselT1Inverse}, is equivalent to $c < r b -
r^2 = \left[ b \right] \left\langle b \right\rangle$.
Hence, $P_1$ is the set of all purely periodic points in $S_1
\setminus R_1$.

There exists no purely periodic point in $S_2$.
In fact, we can see this as follows:
Let $(b, c)$ be an arbitrary element of $P_1$.
We have $T_{2}(b, c) \in S_3$ since $-rb + c + r^2 = c -
\left[ b \right] \left\langle b \right\rangle < 0$
(cf. \eqref{eq:QuadraticHenselT2}).
Hence, no purely periodic orbit enters $S_2$.

Every orbit starting from $(b, c) \in S_3$ passes through $R_4$ and
$P_1$, and then it re-enters $S_3$.
We denote by $(b_0, c_0)$ (resp., $(b_*, c_*)$) the point in $R_4$
(resp., in $P_1$) on the orbit.
Obviously, there exists $m \in {\mathbb Z}_{> 0}$ such that 
$(b, c) = T_{1}^{-m}(b_0, c_0)$.
Since $T_{2}(b_*, c_*) \in S_3$ is a point on the purely periodic
orbit passing through $(b_0, c_0)$, there exists $m_* \in {\mathbb
Z}_{> 0}$ such that $T_{2}(b_*, c_*) = T_{1}^{-m_*}(b_0, c_0)$.
We easily see that $T_{1}^{-1} \circ T_{2}(b_*, c_*) = (-b_*, c_*)$.
Since $b_* > p$ and $c_* > 0$, we see that $T_{1}^{-1} \circ
T_{2}(b_*, c_*) \in S_2$, which implies $m \le m_*$.
Hence, $(b, c)$ is a point on the purely periodic orbit passing
through $(b_0, c_0)$.
Therefore, every $(b, c) \in S_3$ is a purely periodic point.
\end{proof}

\section{Expansions of quadratic elements of ${\mathbb Q}_p$}\label{sec:ExpansionsQuadraticElementsQp}

Let $\alpha$ be an arbitrary quadratic element of ${\mathbb Q}_p$.
In this section, we will first show that each of the three algorithms gives
an eventually periodic expansion for $\alpha$, by showing that the
fractional part $\left\langle \alpha \right\rangle$ is mapped to a
quadratic Hensel root under some iterate of $T_{1}$ and
$T_{2}$.
We will then give a theorem that characterizes elements having purely
periodic expansions for each algorithm.

Let $\alpha^{\sigma}$ be the conjugate of $\alpha$ other than $\alpha$.
We consider the following six cases:
\begin{enumerate}[{Case} 1]
\item
  \begin{enumerate}[A:]
  \item $\lvert \alpha \rvert_p < \lvert \alpha^{\sigma} \rvert_p$ and $\lvert \alpha \rvert_p \le p^{-1}$,\label{Case1A}
  \item $\lvert \alpha \rvert_p < \lvert \alpha^{\sigma} \rvert_p$ and $\lvert \alpha \rvert_p \ge 1$,\label{Case1B}
  \end{enumerate}
\item
  \begin{enumerate}[A:]
  \item $\lvert \alpha \rvert_p > \lvert \alpha^{\sigma} \rvert_p$ and $\lvert \alpha \rvert_p \le p^{-1}$,\label{Case2A}
  \item $\lvert \alpha \rvert_p > \lvert \alpha^{\sigma} \rvert_p$ and $\lvert \alpha \rvert_p \ge 1$,\label{Case2B}
  \end{enumerate}
\item
  \begin{enumerate}[A:]
  \item $\lvert \alpha \rvert_p = \lvert \alpha^{\sigma} \rvert_p$ and $\lvert \alpha \rvert_p \le p^{-1}$,\label{Case3A}
  \item $\lvert \alpha \rvert_p = \lvert \alpha^{\sigma} \rvert_p$ and $\lvert \alpha \rvert_p \ge 1$.\label{Case3B}
  \end{enumerate}
\end{enumerate}

\ \\\noindent
{\it Case~\ref{Case1A}: $\lvert \alpha \rvert_p < \lvert \alpha^{\sigma} \rvert_p$ and $\lvert \alpha \rvert_p \le p^{-1}$}

Let $aX^2 + bX + c \in {\mathbb Z}[X]$ be the minimal polynomial of $\alpha$.
We see that
\begin{align*}
\lvert \alpha^{\sigma} \rvert_p &= \left\lvert - \frac{b}{a} - \alpha \right\rvert_p
 = \left\lvert \frac{b}{a} \right\rvert_p\\
\intertext{and}
\lvert \alpha \rvert_p &= \left\lvert \frac{c}{a \alpha^{\sigma}} \right\rvert_p
 = \left\lvert \frac{c}{b} \right\rvert_p.
\end{align*}
Since $\lvert \alpha \rvert_p < \lvert \alpha^{\sigma} \rvert_p$, we have
\begin{equation}\label{eq:CoefficientRelation}
\left\lvert \frac{a c}{b^2} \right\rvert_p < 1.
\end{equation}
Recalling the definition of $u$, we see that
\begin{equation*}
\frac{u(\alpha) p^{v_p(\alpha)}}{\alpha}=
\frac{u(\alpha)}{\alpha \lvert \alpha \rvert_p}=
\frac{c  \lvert b \rvert_p}{\alpha} \in {\mathbb Z}_p \setminus p{\mathbb Z}_p.
\end{equation*}
By substituting $c \lvert b \rvert_p / X$ for $X$ in $aX^2 + bX + c$,
we have the minimal polynomial of $c \lvert b \rvert_p / \alpha$ given by
\begin{equation*}
X^{2} +  b \lvert b \rvert_p X + a c \lvert b^2 \rvert_p \in {\mathbb Z}[X]. 
\end{equation*}
Note that $b \lvert b \rvert_p \in {\mathbb Z} \setminus p{\mathbb
Z}$, and $a c \lvert b^2 \rvert_p \in p{\mathbb Z}$ by
\eqref{eq:CoefficientRelation}.
Hence, $c \lvert b \rvert_p / \alpha$ is the conjugate of the
quadratic Hensel root $c \lvert b \rvert_p / \alpha^{\sigma}$.
By the proof of \lemref{T1AndT2}, we can see that the fractional part
of the conjugate of a quadratic Hensel root is also a quadratic Hensel
root.
Therefore, $T_{1}(\alpha)= \left\langle c \lvert b \rvert_p / \alpha
\right\rangle$ and $T_{2}(\alpha)= \left\langle - c \lvert b \rvert_p
/ \alpha \right\rangle$ are quadratic Hensel roots.

Consequently, in Case~\ref{Case1A}, $\alpha$ is mapped to a quadratic
Hensel root under one iteration of either $T_{1}$ or $T_{2}$.

\ \\\noindent
{\it Case~\ref{Case1B}: $\lvert \alpha \rvert_p < \lvert \alpha^{\sigma} \rvert_p$ and $\lvert \alpha \rvert_p \ge 1$}

Obviously, $\lvert \alpha \rvert_p = \lvert \left[ \alpha \right] \rvert_p < \lvert
\alpha^{\sigma} \rvert_p$ and $\lvert \alpha^{\sigma} \rvert_p > 1$
hold.
Hence, we have
\begin{equation*}
  \lvert \left\langle \alpha \right\rangle^{\sigma} \rvert_p =
  \lvert \alpha^{\sigma} -  \left[ \alpha \right] \rvert_p =
    \lvert \alpha^{\sigma} \rvert_p > 1.
\end{equation*}
Since $\lvert \left\langle \alpha \right\rangle \rvert_p \le p^{-1}$
and $\lvert \left\langle \alpha \right\rangle \rvert_p < \lvert
\left\langle \alpha \right\rangle^{\sigma} \rvert_p$,
Case~\ref{Case1B} reduces to Case~\ref{Case1A}.

\ \\\noindent
{\it Case~\ref{Case2A}: $\lvert \alpha \rvert_p > \lvert \alpha^{\sigma} \rvert_p$ and $\lvert \alpha \rvert_p \le p^{-1}$}

Let $aX^2 + bX + c \in {\mathbb Z}[X]$ be the minimal polynomial of $\alpha$.
Following a discussion similar to the one in Case~\ref{Case1A}, we see that
\begin{equation}\label{eq:Case2A}
  \lvert \alpha \rvert_p = \left\lvert \frac{b}{a} \right\rvert_p, \quad
  \lvert \alpha^{\sigma} \rvert_p = \left\lvert \frac{c}{b} \right\rvert_p, \quad
  \text{and} \quad
  \left\lvert \frac{b^2}{a c} \right\rvert_p > 1.
\end{equation}

Let us consider the case of applying $T_{1}$ to $\alpha$.
Since
\begin{align*}
T_{1}(\alpha) &=\frac{u(\alpha)}{\alpha \lvert \alpha \rvert_p} - d_{1}(\alpha),\\
\intertext{the conjugate $T_{1}(\alpha)^{\sigma} \neq T_{1}(\alpha)$ of $T_{1}(\alpha)$ is given by}
T_{1}(\alpha)^{\sigma} &=\frac{u(\alpha)}{\alpha^{\sigma} \lvert \alpha \rvert_p} - d_{1}(\alpha).
\end{align*}
By \eqref{eq:Case2A}, we have
\begin{equation*}
  \left\lvert \frac{u(\alpha)}{\alpha^{\sigma} \lvert \alpha \rvert_p} \right\rvert_p
  = \left\lvert \frac{b^2}{a c} \right\rvert_p > 1.
\end{equation*}
Since $\lvert d_{1}(\alpha) \rvert_p = 1$, we have
\begin{equation*}
\left\lvert T_{1}(\alpha)^{\sigma} \right\rvert_p =
  \left\lvert \frac{u(\alpha)}{\alpha^{\sigma} \lvert \alpha \rvert_p} \right\rvert_p
  > 1.
\end{equation*}

Similarly, in the case of applying $T_{2}$ to $\alpha$, we see
that $\left\lvert T_{2}(\alpha)^{\sigma} \right\rvert_p > 1$.

Since $\left\lvert T_{1}(\alpha) \right\rvert_p \le p^{-1}$ and
$\left\lvert T_{2}(\alpha) \right\rvert_p \le p^{-1}$,
we see that
$\left\lvert T_{1}(\alpha) \right\rvert_p < \left\lvert
T_{1}(\alpha)^{\sigma} \right\rvert_p$
and
$\left\lvert T_{2}(\alpha) \right\rvert_p < \left\lvert
T_{2}(\alpha)^{\sigma} \right\rvert_p$.
Therefore, Case~\ref{Case2A} reduces to Case~\ref{Case1A} after one
iteration of $T_{1}$ or $T_{2}$.

\ \\\noindent
{\it Case~\ref{Case2B}: $\lvert \alpha \rvert_p > \lvert \alpha^{\sigma} \rvert_p$ and $\lvert \alpha \rvert_p \ge 1$}

Since $\lvert \left[ \alpha \right] \rvert_p > \lvert
\alpha^{\sigma} \rvert_p$ and $\lvert \left[ \alpha \right] \rvert_p \ge 1$,
we have
\begin{equation*}
  \lvert \left\langle \alpha \right\rangle^{\sigma} \rvert_p =
  \lvert \alpha^{\sigma} -  \left[ \alpha \right] \rvert_p =
    \lvert \left[ \alpha \right] \rvert_p \ge 1.
\end{equation*}
Since $\lvert \left\langle \alpha \right\rangle \rvert_p \le p^{-1}$
and $\lvert \left\langle \alpha \right\rangle \rvert_p < \lvert
\left\langle \alpha \right\rangle^{\sigma} \rvert_p$,
Case~\ref{Case2B} reduces to Case~\ref{Case1A}.

\ \\\noindent
{\it Case~\ref{Case3A}: $\lvert \alpha \rvert_p = \lvert \alpha^{\sigma} \rvert_p$ and $\lvert \alpha \rvert_p \le p^{-1}$}

Let $\left\{ \epsilon_n \right\}_{n \ge 1}$ be an arbitrary
sequence in the set $\left\{1, 2\right\}$.
We obtain an expansion of $\alpha$ of the form
\eqref{eq:InfiniteContinuedFraction} by applying $T_{\epsilon_n} \circ
\dots \circ T_{\epsilon_1}$ ($n \in {\mathbb Z}_{> 0}$) to $\alpha$.
Let us define a sequence $\left\{ \alpha_{n} \right\}_{n \ge 0}$ by
\begin{equation*}
\alpha_{0}=\alpha \quad \text{and} \quad \alpha_{n}= T_{\epsilon_n} \circ \dots \circ T_{\epsilon_1}(\alpha_{0}) \quad (n \ge 1).
\end{equation*}
Assuming that $\lvert \alpha_n \rvert_p = \lvert \alpha_n^{\sigma}
\rvert_p$ for all $n \in {\mathbb Z}_{\ge 0}$, it is not difficult to see
that the expansion of $\alpha$ is identical with that of
$\alpha^{\sigma}$ obtained by applying $T_{\epsilon_n} \circ \dots
\circ T_{\epsilon_1}$ ($n \in {\mathbb Z}_{> 0}$) to
$\alpha^{\sigma}$.
Then, by \thmref{ConvergenceTheorem} (ii), we get $\alpha=\alpha^{\sigma}$,
which is a contradiction.
This proves that there exists $n \in {\mathbb Z}_{> 0}$ such that
$\lvert \alpha_n \rvert_p \neq \lvert \alpha_n^{\sigma} \rvert_p$.
Therefore, Case~\ref{Case3A} reduces to Case~\ref{Case1A} or
Case~\ref{Case2A} after sufficient iterations of $T_{1}$ and $T_{2}$.


\ \\\noindent
{\it Case~\ref{Case3B}: $\lvert \alpha \rvert_p = \lvert \alpha^{\sigma} \rvert_p$ and $\lvert \alpha \rvert_p \ge 1$}

If $\lvert \left\langle \alpha \right\rangle \rvert_p \neq \lvert
\left\langle \alpha \right\rangle^{\sigma} \rvert_p$,
Case~\ref{Case3B} reduces to Case~\ref{Case1A} or Case~\ref{Case2A};
otherwise Case~\ref{Case3B} reduces to Case~\ref{Case3A}.
\ \\

Consequently, in all the cases, $\left\langle \alpha \right\rangle$ is
mapped to a quadratic Hensel root under some iterate of $T_{1}$ and
$T_{2}$, regardless of the order in which they are applied.
Hence, by Theorems~\ref{QuadraticHenselAlgorithmA}, \ref{QuadraticHenselAlgorithmB}, and \ref{QuadraticHenselAlgorithmC}, we have the
following theorem.

\begin{thm}\label{PadicLagrangeTheorem}
The expansion of every quadratic element of ${\mathbb Q}_p$ over
${\mathbb Q}$ obtained by each of Algorithms A, B, and C is
eventually periodic.
\end{thm}

We now turn to the characterization of elements with purely periodic
expansions for each algorithm.
Note that we consider $d_{0}$ in
\eqref{eq:IntroInfiniteContinuedFraction} to be zero for purely
periodic expansions, and therefore elements with purely periodic
expansions are necessarily in $p{\mathbb Z}_p$.

Except for quadratic Hensel roots, there exists no quadratic element
of $p{\mathbb Z}_p$ whose expansion by our algorithms is purely
periodic.
This is because every quadratic element of $p{\mathbb Z}_p$ is mapped
to a quadratic Hensel root under some iterate of $T_{1}$ and $T_{2}$,
as we have seen above.

By \thmref{RationalNumberTheorem}, which will be shown in
Section~\ref{sec:ExpansionsRationalNumbers}, we also see that there
exists no rational number whose expansion is periodic.

Consequently, for each algorithm, the set of elements having purely
periodic expansions, i.e., the reduced set, is identical with the set
of purely periodic points within the set $S$ of quadratic Hensel
roots.
Hence, by \thmref{QuadraticHenselAlgorithmA},
\corref{PurelyPeriodicPointsAlgorithmB}, and
\lemref{PurelyPeriodicPointsAlgorithmC}, we have the following
theorem.

\begin{thm}\label{PadicGaloisTheorem}
The reduced sets for Algorithms A, B, and C are given by
$S$, $R$, and $P_1 \cup R_1 \cup S_3 \cup S_4$, respectively.
\end{thm}

\section{Expansions of rational Hensel roots}\label{sec:ExpansionsRationalHenselRoots}

$\alpha=0$ is the root of $X \in {\mathbb Z}[X]$ and thus is a
rational Hensel root.
As described in Section~\ref{sec:PadicContinuedFractionExpansion},
we do not expand $\alpha=0$ any further.

Let us consider expansions of rational Hensel roots other
than $0$, i.e., those of roots of $X + c \in {\mathbb
Z}[X]$ with $c \in p{\mathbb Z} \setminus \left\{ 0 \right\}$.
Recall the definitions of our algorithms in
Section~\ref{sec:ContinuedFractionAlgorithms}.
Since the coefficient $b$ of $X$ of the minimal polynomial in ${\mathbb Z}[X]$ satisfies $b=1$
for every rational Hensel root, we can classify our algorithms into two
classes:
\begin{enumerate}[{Class} I:]
\item Algorithms which apply $T_{2}$ to every rational Hensel root other than $0$
  (Algorithms A and B)\label{ClassI}
\item Algorithms which apply $T_{2}$ to a rational Hensel root if the
  coefficient $c$ of its minimal polynomial satisfies $c>0$ and apply
  $T_{1}$ if $c<0$
  (Algorithm C)\label{ClassII}
\end{enumerate}
In the following, we show that whichever class an algorithm belongs
to, it gives a finite expansion for every rational Hensel root
other than $0$.

\ \\\noindent
{\it Class~\ref{ClassI}}

Let $\alpha$ be an arbitrary rational Hensel root other than $0$.
Applying $T_{2}$ to $\alpha$, we have
\begin{equation*}
T_{2}(\alpha) =\frac{-u(\alpha)}{\alpha \lvert \alpha \rvert_p} - d_{2}(\alpha).
\end{equation*}
Let $X + c \in {\mathbb Z}[X]$ be the minimal polynomial of $\alpha$.
Since $\alpha=-c$, we see that $\lvert \alpha \rvert_p = \lvert c \rvert_p$.
Since
\begin{equation*}
\frac{-u(\alpha)}{\alpha \lvert \alpha \rvert_p}= 1,
\end{equation*}
we have $d_{2}(\alpha)=1$, which implies
\begin{equation*}
T_{2}(\alpha) =\frac{-c}{\alpha} - 1 = 0.
\end{equation*}
Therefore, $\alpha$ is expanded into the finite continued fraction
\begin{equation*}
\alpha=\frac{-c}{1}
\end{equation*}
by each of the algorithms belonging to Class~\ref{ClassI}.

\ \\\noindent
{\it Class~\ref{ClassII}}

Let $\alpha$ be an arbitrary rational Hensel root other than $0$, and
let $X + c \in {\mathbb Z}[X]$ be the minimal polynomial of $\alpha$.

If $c>0$, we apply $T_{2}$ to $\alpha$.
Hence, $\alpha$ is expanded as
\begin{equation*}
\alpha=\frac{-c}{1}
\end{equation*}
(cf. Class~\ref{ClassI}).

Let us consider the case where $c<0$.
In this case, we apply $T_{1}$ to $\alpha$.
Since
\begin{equation*}
\frac{u(\alpha)}{\alpha \lvert \alpha \rvert_p}= -1,
\end{equation*}
we see that $d_{1}(\alpha)=p-1$, which implies
\begin{equation}\label{eq:ClassIIc<0part1}
T_{1}(\alpha) =\frac{c}{\alpha} - (p-1) = -p.
\end{equation}
Note that $T_{1}(\alpha) = -p$ is also a rational Hensel root whose
minimal polynomial is $X + p \in {\mathbb Z}[X]$.
Since the constant term $p$ of $X + p$ is positive, $-p$ is expanded as
\begin{equation}\label{eq:ClassIIc<0part2}
-p=\frac{-p}{1}
\end{equation}
by using $T_{2}$.
By \eqref{eq:ClassIIc<0part1} and \eqref{eq:ClassIIc<0part2}, we
see that $\alpha$ is expanded as
\begin{equation*}
\alpha=\cfrac{c}{p-1 +
    \cfrac{-p}{1}}
\end{equation*}
when $c<0$.

Therefore, the algorithm belonging to
Class~\ref{ClassII} also gives a finite continued fraction for
$\alpha$.
\ \\

As a consequence, we get the following theorem.

\begin{thm}\label{RationalHenselRootTheorem}
Each of Algorithms A, B, and C gives a finite expansion for every
rational Hensel root.
\end{thm}

\section{Expansions of rational numbers}\label{sec:ExpansionsRationalNumbers}

In this section, we will show that each of the three algorithms gives
a finite expansion for every rational number.

Let $\alpha$ be an arbitrary rational number.
We distinguish the following two cases:
  \begin{enumerate}[{Case} A:]
  \item $\lvert \alpha \rvert_p \le p^{-1}$,\label{CaseA}
  \item $\lvert \alpha \rvert_p \ge 1$.\label{CaseB}
  \end{enumerate}

\ \\\noindent
{\it Case~\ref{CaseA}: $\lvert \alpha \rvert_p \le p^{-1}$}

If $\alpha$ is a rational Hensel root, our algorithms give a finite
expansion for $\alpha$ (cf. \thmref{RationalHenselRootTheorem}).

Let us consider the case where $\alpha$ is not a rational Hensel root.
Let $bX + c \in {\mathbb Z}[X]$ be the minimal polynomial of $\alpha$.
Obviously, $\alpha=-c/b$.
Since $b$, $c$ are coprime and $\lvert \alpha \rvert_p \le p^{-1}$, we
see that $b \in {\mathbb Z} \setminus p{\mathbb Z}$ and $c \in
p{\mathbb Z}$.
Hence, we have $\lvert \alpha \rvert_p= \lvert c \rvert_p$.
We easily see that
\begin{equation*}
\frac{u(\alpha)}{\alpha \lvert \alpha \rvert_p}=
-b \in {\mathbb Z} \setminus p{\mathbb Z},
\end{equation*}
which implies that $T_{1}(\alpha)= \left\langle -b \right\rangle$ and
$T_{2}(\alpha)= \left\langle b \right\rangle$ are in $p{\mathbb Z}$,
i.e., they are rational Hensel roots.
Therefore, the expansion of $\alpha$ obtained by each of our
algorithms is finite also in the case where $\alpha$ is not a rational
Hensel root.

\ \\\noindent
{\it Case~\ref{CaseB}: $\lvert \alpha \rvert_p \ge 1$}

Since $\left\langle \alpha \right\rangle \in {\mathbb Q}$ and $\lvert
\left\langle \alpha \right\rangle \rvert_p \le p^{-1}$,
Case~\ref{CaseB} reduces to Case~\ref{CaseA}.
\ \\

Summarizing the discussion above, we have the following theorem.

\begin{thm}\label{RationalNumberTheorem}
Each of Algorithms A, B, and C gives a finite expansion for every
rational number.
\end{thm}

\section{Concluding remarks}\label{sec:ConcludingRemarks}

It is worth making some remarks on the generality of our results.
We denote by $A_p'$ the set of all the algebraic elements of ${\mathbb
Q}_p$ of degree at most two.

\begin{enumerate}
\item
We have dealt with continued fractions of the form
\eqref{eq:IntroInfiniteContinuedFraction}, but we can also consider
another basic type, namely continued fractions of the form
\begin{equation}\label{eq:MultiplicativeIntroInfiniteContinuedFraction2}
  \cfrac{t_{1} p^{k_{1}}}{d_{1}+
    \cfrac{t_{2} p^{k_{2}}}{d_{2}+
      \cfrac{t_{3} p^{k_{3}}}{
     \ddots}}} \quad (k_1 \in {\mathbb Z}, ~ k_{n+1} \in {\mathbb Z}_{> 0},
  ~ t_{n}, d_{n} \in {\mathbb Z} \setminus p{\mathbb Z} ~ (n \ge 1)).
\end{equation}
The convergence of continued fractions
\eqref{eq:MultiplicativeIntroInfiniteContinuedFraction2} is also
guaranteed by \thmref{ConvergenceTheorem}.
By replacing the domain $A_p \setminus \left\{ 0 \right\}$ of $T_{1}$
and $T_{2}$ by $A_p' \setminus \left\{ 0 \right\}$, we can modify our
algorithms to generate continued fractions of the form
\eqref{eq:MultiplicativeIntroInfiniteContinuedFraction2}.
Even with this modification, all the theorems in this paper still hold
for the modified algorithms.

\item
We have focused on the expansion of the elements of $A_p'$, but it is easy to
extend our algorithms to cover all the elements of ${\mathbb Q}_p$.
One of the simplest ways to do this is to expand every element of
${\mathbb Q}_p \setminus A_p'$ by using the map $T$ such that $t$ and
$d$ in \eqref{eq:T} satisfy $t(x) \equiv 1$ and $\text{Im}(d) \subset
\left\{1, \dots, p-1 \right\}$.
The convergence of resulting continued fractions is guaranteed, as we
have seen in Section~\ref{sec:ConvergenceContinuedFractions}.
Note that even with this extension, a given element of ${\mathbb Q}_p
\setminus A_p'$ has neither a periodic nor finite expansion since
$\text{Im}(t)$ and $\text{Im}(d)$ are included in ${\mathbb Q}$.
\end{enumerate}

Algorithms other than those presented here, as well as the extension
of our approach to multidimensional $p$-adic continued fractions, will
be reported in forthcoming papers.

\section*{Acknowledgements}

This research was supported by JSPS KAKENHI Grant Number 15K00342.


\begin{thebibliography}{100}
\bibitem{Browkin}J. Browkin,
Continued fractions in local fields, II,
Math. Comp. 70 (2001), 1281--1292.

\bibitem{B}P. Bundschuh,
$p$-adische Kettenbr\"uche und Irrationalit\"at $p$-adischer Zahlen,
Elem. Math. 32 (1977), 36--40.

\bibitem{FISTY}M. Furukado, S. Ito, A. Saito, J.-I. Tamura, S. Yasutomi,
A new multidimensional slow continued fraction algorithm and stepped surface,
Experimental Math. 23 (2014), 390--410.

\bibitem{Galois}\'{E}. Galois,
D\'{e}monstration d'un th\'{e}or\`{e}me sur les fractions continues p\'{e}riodiques,
Annales de math\'{e}matiques pures et appliqu\'{e}es 19 (1828/29), 294-–301.

\bibitem{Lagrange}J.-L. Lagrange,
Additions au m\'{e}moire sur la r\'{e}solution des \'{e}quations num\'{e}riques,
M\'{e}m. Berl. 24 (1770).

\bibitem{O}T. Ooto, 
Transcendental $p$-adic continued fractions, preprint.

\bibitem{Perron}O. Perron,
{\it Die Lehre von den Kettenbr\"uchen}
(Teubner, Leipzig, 1913).

\bibitem{R}A. A. Ruban,
Certain metric properties of $p$-adic numbers (Russian),
Sibirsk. Mat. Zh. 11 (1970), 222--227.

\bibitem{S}T. Schneider,
\"Uber $p$-adische Kettenbr\"uche,
Symp. Math. 4 (1968/69), 181--189.

\bibitem{Tamura}J.-I. Tamura,
A $p$-adic phenomenon related to certain integer matrices, and $p$-adic values of a multidimensional continued fraction,
in {\it Summer School on the Theory of Uniform Distribution},
RIMS K\^{o}ky\^{u}roku Bessatsu B29 (2012), 1–-40. 

\bibitem{TY1} J.-I. Tamura, S. Yasutomi,
A new multidimensional continued fraction algorithm,
Math. Comp. 78 (2009), 2209--2222.

\bibitem{TY2} J.-I. Tamura, S. Yasutomi, 
Algebraic Jacobi-Perron algorithm for biquadratic numbers,
in {\it Diophantine Analysis and Related Fields 2010},
AIP Conf. Proc. 1264 (2010), 139--149.

\bibitem{TY3} J.-I. Tamura, S. Yasutomi,
A new algorithm of continued fractions related to real algebraic number fields of degree $\le$ 5,
Integers 11B (2011), A16.

\bibitem{TY4} J.-I. Tamura, S. Yasutomi,
Some aspects of multidimensional continued fraction algorithms,
in {\it Functions in Number Theory and Their Probabilistic Aspects},
RIMS K\^{o}ky\^{u}roku Bessatsu B34 (2012), 463--475.

\bibitem{W}B. M. M. de Weger,
Periodicity of $p$-adic continued fractions,
Elem. Math. 43 (1988), 112--116.

\bibitem{Weger2}B. M. M. de Weger,
Approximation lattices of $p$-adic numbers,
J. Number Theory 24 (1986), 70--88.
  
\end{thebibliography}
\end{document}